\documentclass[a4paper,10pt]{article}
\usepackage[utf8x]{inputenc}
\usepackage{amssymb}
\usepackage{multicol}
\usepackage{enumitem}
\usepackage{amsmath}
\usepackage{amsthm}
\usepackage{mathrsfs}
\usepackage{fullpage}
\usepackage{cite}
\usepackage[english]{babel}
\usepackage{nicefrac}

\title{Quantitative results for Bruck iterations of demicontinuous pseudocontractions}
\author{ Daniel K\"ornlein\thanks{The 
author has been 
supported by the German Science Foundation (DFG 
Project KO 1737/5-2).}\\[0.2cm]
Department of Mathematics \\  Technische Universit\"at Darmstadt\\ 
Schlossgartenstra\ss{}e 7, 64289 Darmstadt, Germany}

\newtheorem{theorem}{Theorem}[section]

\theoremstyle{definition}

\newtheorem{lemma}[theorem]{Lemma}
\newtheorem{definition}[theorem]{Definition}

\theoremstyle{remark}

\newtheorem{remark}[theorem]{Remark}

\newcommand{\NN}{\mathbb{N}}
\newcommand{\RR}{\mathbb{R}}

\newcommand{\diam}{\operatorname{diam}}

\begin{document}

\maketitle
\begin{abstract}
Our first result is a rate of metastability in the sense of Tao for Bruck's iteration scheme for demicontinuous pseudocontractions in Hilbert space, extracted from Bruck's original proof. This result generalizes earlier work in the ongoing program of proof mining from Lipschitzian to demicontinuous pseudocontractions. Our second main result is a metastable version of asymptotic regularity under the additional assumption that the underlying operator is norm-to-norm uniformly continuous on bounded subsets. These results (and their intermediate versions given in this paper) provide a thorough quantitative analysis of Bruck's iteration scheme for pseudocontractions in Hilbert space.
\end{abstract}

\let\thefootnote\relax

\section{Introduction}
Let $X$ be a normed linear space and $S\subseteq X$ be a subset of $X.$ 
In 1967, Browder introduced an important generalization of the 
class of nonexpansive mappings, namely the {\em pseudocontractive} 
mappings $T:S\to S$ defined by 
\[ \forall u,v\in S\,\forall \lambda >1 \ ((\lambda-1)\| u-v\| 
\le \| (\lambda I-T)(u)-(\lambda I-T)(v)\|),\] where $I$ denotes the 
identity mapping.  

Apart from being a generalization of nonexpansive mappings, the 
pseudocontractive mappings are also closely related to accretive operators, 
where an operator $A$ is called accretive if for every 
$u,v\in S$ and for all $s>0$,
\begin{equation*}
 \left\Vert u-v\right\Vert\leq\left\Vert u-v+s\left(Au-Av\right)\right\Vert.
\end{equation*}
Observe that $T$ is pseudocontractive if and only if $I-T$ is accretive. 
Therefore, any fixed point of $T$ is a root of the accretive operator $I-T$. 

In a Hilbert space, $T$ is pseudocontractive if and only if
\[ \forall u,v\in S \,(\langle Tu-Tv,u-v\rangle \le \| u-v\|^2) \] 
(see e.g. \cite{Chidume(09)}).
 
In \cite{Bruck(74)}, Bruck introduced the following iteration schema 
for pseudocontractive mappings:
 \begin{definition}[\cite{Bruck(74)}]
     Let $C$ be a nonempty convex subset of a real normed space and 
let $T:C\to C$ be a pseudocontraction. Let $(\lambda_n),(\theta_n)$ be 
sequences in $[0,1]$ with $\lambda_n(1+\theta_n)\le 1$ for all $n\in\NN.$  
The 
\textbf{Bruck iteration scheme} with starting point $x_1\in C$ is defined as
     \begin{equation*}
      x_{n+1}=\left(1-\lambda_n\right)x_n+\lambda_nTx_n-\lambda_n\theta_n
\left(x_n-x_1\right).
     \end{equation*}
    \end{definition}
Among many other things, Bruck showed that in Hilbert spaces and for 
bounded closed and convex subsets $C$ this iteration 
strongly converges for so-called acceptably paired sequences 
$(\lambda_n),(\theta_n)$ (cf. Definition \ref{def:acceptably}). 
Moreover the limit is a fixed point of $T$ provided that $T$ is demicontinuous (continuous from the strong to the weak topology on $H$) in addition to being pseudocontractive:
    \begin{definition}[{\cite{Bruck(74)}}]\label{def:acceptably}
      Two sequences $\left(\lambda_n\right)$ and $\left(\theta_n\right)$ in $[0,1]$ are \textbf{acceptably paired} if $\left(\theta_n\right)$ is nonincreasing, $\lim_{n\to\infty}\theta_n=0$ and there exists a strictly increasing sequence $\left(f(n)\right)_n$ of positive integers such that
      \begin{enumerate}
       \item $\liminf\limits_{n\to\infty}\theta_{f\left(n\right)}\cdot\sum\limits_{j=f\left(n\right)}^{f\left(n+1\right)}\lambda_j>0$,
       \item $\lim\limits_{n\to\infty}\left(\theta_{f\left(n\right)}-\theta_{f\left(n+1\right)}\right)\cdot\sum\limits_{j=f\left(n\right)}^{f\left(n+1\right)}\lambda_j=0$, and
       \item $\lim\limits_{n\to\infty}\sum\limits_{j=f\left(n\right)}^{f\left(n+1\right)}\lambda_j^2=0$.
      \end{enumerate}
    \end{definition}
\begin{theorem}[Corollary 4 of {\cite{Bruck(74)}}]\label{PseudoMonotone}
      Let $C$ be a nonempty bounded closed convex subset of a Hilbert space $H$ and $T:C\to C$ be a demicontinuous pseudocontraction. If $\left(\lambda_n\right)$ and $\left(\theta_n\right)$ are acceptably paired such that $\lambda_n\left(1+\theta_n\right)\le1$, then, for all $x_1,z\in C$, the sequence $\left(x_n\right)$ defined by
      \begin{equation*}
       x_{n+1}=\left(1-\lambda_n\right)x_n+\lambda_nTx_n+\lambda_n\theta_n\left(z-x_n\right)
      \end{equation*}
      remains in $C$ and converges strongly to the fixed point of $T$ which is closest to $z$.
    \end{theorem}
    
    Effective uniform rates on the strong convergence of $(x_n)$ are generally ruled out. In fact, Neumann \cite{Neumann(15)} showed that there are (computable) nonexpansive mappings $f$ on the Hilbert cube (sequences $(x_n)\in\ell_2$ with $\vert x_n\vert\le1$ for all $n$) that have no computable fixed points, and so no sequence approximating any fixed point of $f$ can have a computable rate of convergence. Following general proof-theoretic methods, it is necessary to pass first to a finitary version of Cauchyness, the so-called metastability in the sense of Tao, i.e. (here $[n;n+g(n)]:=\{ n,n+1,n+2,\ldots,n+g(n)\}$) 
\begin{equation*}
  \forall \varepsilon >0\,\forall g:\NN\to\NN\,\exists n\in\NN\,\forall i,
  j\in [n;n+g(n)]\,\big(\Vert x_i-x_j\Vert<\varepsilon\big).
\end{equation*}
Metastability is the so-called Herbrand normal form of (a suitable reformulation of) the Cauchy statement for the sequence $(x_n)$, and, as such, is equivalent to the original statement. It is finitary in the sense that it only talks about finite subsequences of $(x_n)$. A rate of metastability is then a bound $\Phi:\NN\times\NN^\NN\to\NN$ on the existential quantifier:
\begin{equation*}
  \forall \varepsilon >0\,\forall g:\NN\to\NN\,\exists n\le\Phi(\varepsilon,g)\,\forall i,j\in [n;n+g(n)]\,\big(\Vert x_i-x_j\Vert<\varepsilon\big).
\end{equation*}
Such bounds are guaranteed to exist under vastly general conditions on the complexity of the proof (\cite{Kohlenbach(08)}). 

A quantitative, finitary version of all of Theorem \ref{PseudoMonotone}, however, should not only finitise the Cauchyness of $(x_n)$, but also that the strong limit is indeed a fixed point. If $T$ were norm-to-norm continuous, one way to do so would be to ensure that the sequence $(x_n)$ is not only Cauchy along the interval $[n;n+g(n)]$, but also asymptotically regular:
\begin{equation*}
 \forall\varepsilon>0\,\forall g:\NN\to\NN\,\exists n\le\Phi(\varepsilon,g)\,\forall i,j\in[n;n+g(n)]\,\big(\Vert x_i-x_j\Vert<\varepsilon\,\wedge\,\Vert Tx_i-x_i\Vert<\varepsilon\big).
\end{equation*}
By the logical equivalence of a statement to its Herbrand normal form, this implies both Cauchyness and asymptotic regularity. Cauchyness then implies that the strong limit exists, while norm-to-norm continuity and asymptotic regularity recover the fact that the limit is indeed a fixed point.

In the case at hand, however, the operator $T$ is only demicontinuous. In fact, convergence to a fixed point is established via the \textit{continuous} path $(z_t)$ defined by $z_t=tTz_t+(1-t)z$, which -- in turn -- converges strongly to the fixed point of $T$ closest to $z$. This gives rise to the following finitary version of Theorem \ref{PseudoMonotone}:
\begin{equation}\label{eq:finite}
 \forall\varepsilon>0\,\forall g:\NN\to\NN\,\exists n\le\Phi(\varepsilon,g)\,\forall i,j\in[n;n+g(n)]\,\big(\Vert x_i-x_j\Vert<\varepsilon\,\wedge\Vert x_i-y_i\Vert<\varepsilon\,\wedge\,\Vert y_i-Ty_i\Vert<\varepsilon\big),\tag{+}
\end{equation}
where $y_i=z_{1/1+\theta_i}$. Our main theorem (Theorem \ref{thm:deliver}) provides such a bound. If $T$ is even norm-to-norm uniformly continuous with modulus $\omega$, then one can obtain a bound $\Delta$ such that (see Theorem \ref{thm:unifcont})
     \begin{equation*}
       \forall\varepsilon>0\forall g:\NN\to\NN\exists n\le\Delta(g,\varepsilon)\forall i,j\in[n;n+g(n)]\bigl(\Vert x_i-x_j\Vert\le\varepsilon\wedge\Vert x_i-Tx_i\Vert\le\varepsilon\bigr).
      \end{equation*}
This is a generalization of Theorem 2.8 of \cite{KoernleinKohlenbach(13)}, which required $T$ to be Lipschitz continuous. As guaranteed by general logical metatheorems \cite{Kohlenbach(08)}, these bounds are highly uniform in the input data; it is independent of the space and the concrete choices for the operator $T:C\to C$, the set $C$ or the parameter sequences $(\lambda_n)$ and $(\theta_n)$. Apart from the counterfunction $g$ and the accuracy $\varepsilon$, the bounds only depend on an upper bound on the diameter $\diam(C)$, moduli for the quantiative version of acceptably pariedness (cf. Definition \ref{def:acceptablyQuant}) and, in the case of Theorem \ref{thm:unifcont}, the modulus of uniform continuity $\omega$.

Moreover, the new, logically transformed proof of \eqref{eq:finite} is totally elementary in that all ideal principles have been eliminated; it can be formalized in constructive (``intuitionistic'') arithmetic enriched by an abstract normed space $X$ (see Kohlenbach \cite{Kohlenbach(08)}) and axioms asserting that $X$ is a Hilbert space. Moreover, one can recover Bruck's original theorem using only the axiom of choice over \textit{quantifier-free} sentences.
\section{Analysis of Bruck's Proof}

We now examine from a proof-theoretic perspective the steps into which Bruck's proof of Theorem \ref{PseudoMonotone} decomposes. First of all, we need to recall the generalization of pseudocontractiveness to set-valued operators. $T\subseteq H\times H$ is pseudocontractive, if, for all $(u,x),(v,y)\in T$,
\begin{equation*}
 \langle x-y,u-v\rangle\le\Vert u-v\Vert^2.
\end{equation*}
Moreover, an operator $U\subset H\times H$ is monotone if and only if $I-U$ is pseudocontractive. It is \textit{maximal} monotone if there does not exist a monotone $U'\subset H\times H$ such that $U\subsetneq U'$.

Bruck's proof then follows the following line of argument:
\begin{enumerate}[label=(\roman*)]
 \item\label{item:1} The monotone operator $U:I-T$ is extended to a maximal monotone, set valued operator $U^*\subset H\times H$.
 \item\label{item:2} There exists a unique $y_{\theta}$ for each $\theta>0$ for which $0\in\theta(y_{\theta}-z)+U^*(y_\theta)$.
 \item\label{item:3} The strong $\lim_{\theta\to0^+}y_{\theta}$ exists and is the point $x^*$ of ${U^*}^{-1}(0)$ closest to $z$.
 \item\label{item:4} The sequence $(x_n)$ also converges to $x^*$.
 \item\label{item:5} The limit is a zero of $U$, and hence a fixed point of $T$.
\end{enumerate}
The existence of a maximal monotone extension of a monotone operator $U:H\to H$ makes use of Zorn's Lemma, which is equivalent to the Axiom of Choice. However, we are, for this paper, only interested in the single-valued case. As shown in \cite{KoernleinKohlenbach(13)}, it is possible to avoid the detour via maximal monotone extensions. A similar result has also been shown by Lan and Wu in \cite{LanWu(02)}. The existence of the path $(y_t)_{t\in(0,1]}$ is also guaranteed in this case since the mapping $U_t:C\to C,y\mapsto t(y-z)+U(x)$ is $t$-strongly monotone for each $t>0$, and thus has a unique fixed point (see \cite{KoernleinKohlenbach(13)}). The mere existence of the sequence $(y_{\theta_n})_n$ makes no proof-theoretic contribution since their defining property is a purely universal statement, i.e. one with only $\forall$-quantifiers.

The  convergence of $(y_{\theta_n})$ to the fixed point of $T$ closest to $z$ is then carried out analogously to the multi-valued case in Bruck's proof \cite{Bruck(74)}. A quantitative analysis of this step has been performed and a rate of metastability has already been extracted in \cite{KoernleinKohlenbach(13)}.

The convergence $\Vert x_n-x^*\Vert\to0$ is established via convergence of the subsequence $\Vert x_{f(n)}-x^*\Vert\to0$, which is shown using the existence of the limit superior as a translation invariant functional $\limsup:\ell_\infty\to\RR$ as follows: If $f:\NN\to\NN$ denotes the subsequence from Definition \ref{def:acceptably}, then there exists a constant $\gamma\in(0,1)$ such that
\begin{align}
 \gamma\cdot\limsup\Vert x_{f(k)}-x^*\Vert^2&=\gamma\cdot\limsup\Vert x_{f(k)}-y_{\theta_{f(k)}}\Vert^2\notag\\
					    &\ge\limsup\Vert x_{f(k+1)}-y_{\theta_{f(k)}}\Vert^2\label{eq:!}\\
					    &=\limsup\Vert x_{f(k+1)}-x^*\Vert^2\notag\\
					    &=\limsup\Vert x_{f(k)}-x^*\Vert^2.\notag
\end{align}
where inequality \eqref{eq:!} is shown in Bruck's proof. Therefore, $\limsup\Vert x_{f(k)}-x^*\Vert=0$, so the subsequence $(x_{f(k)})$ converges to $x^*$. Basic arithmetic then implies the convergence of the original sequence.

\section{Main Results}

To obtain a quantitative version of Theorem \ref{PseudoMonotone}, we need a quantitative version of what it means for two sequences to be acceptably paired. 
    \begin{definition}\label{def:acceptablyQuant} Two sequences $(\lambda_n)$ and $(\theta_n)$ in $[0,1]$ are called acceptably paired with moduli $\varphi_1,\varphi_2,\varphi_3:\RR\to\NN$, $f:\NN\to\NN$, $n_0\in\NN$ and $\delta>0$ if  $(\theta_n)$ is nonincreasing and the following conditions are satisfied:
     \begin{enumerate}
      \item\label{i} $\forall\varepsilon>0\forall n\ge\varphi_1(\varepsilon)\left(\theta_n\le\varepsilon\right)$,
      \item\label{ii} $\forall n(f(n+1)\ge f(n)+1)$,
      \item\label{iii} $\forall n\ge n_0 \left(\theta_{f(n)}\cdot\sum\limits_{j=f(n)}^{f\left(n+1\right)}\lambda_j\ge\delta\right)$,
      \item\label{iv} $\forall\varepsilon>0\forall n\ge\varphi_2(\varepsilon)\left((\theta_{f(n)}-\theta_{f(n+1)})\cdot\sum\limits_{j=f(n)}^{f\left(n+1\right)}\lambda_j\le\varepsilon\right)$, and
      \item\label{v} $\forall\varepsilon>0\forall n\ge\varphi_3(\varepsilon)\left(\sum\limits_{j=f(n)}^{f(n+1)}\lambda_j^2\le\varepsilon\right)$.
     \end{enumerate}
    \end{definition}
  The moduli $\varphi_i$ are rates of convergence of their respective sequences to $0$. The numbers $n_0$ and $\delta$ are quantitative witnesses for the condition that the sequence $\theta_{f(n)}\cdot\sum_{i=f(n)}^{f(n+1)}\lambda_j$ stays strictly away from $0$, i.e.~its $\liminf$ is greater than $0$. It is also noteworthy that the function $k\mapsto k^*:=\max\{n\in\NN:f(n)\le k\}$ is well-defined for all $k\ge f(0)$. Moreover, $(f(k))^*=k$ for all nonnegative integers $k$.
    \begin{remark}[{\cite{Bruck(74)}}]\label{rem:ex}
    Examples of acceptably paired sequences are:
     \begin{enumerate}
      \item $\lambda_n=1/n$, $\theta_n=1/\log\log n$ and $f(n)=n^n$.
      \item For $0<p<1$ and $0<q<\min\left\{p,1-p\right\}$, $\lambda_n=n^{-p}$ and $\theta_n=n^{-q}$ are acceptably paired with $f(n)=\lceil n^{d/(1-p)}\rceil$ for suitable $d>1$ (see Section \ref{sec:ex1} for details).
     \end{enumerate}
     The corresponding moduli will be given in Section \ref{sec:ex}.
    \end{remark}

%
%
%
%
%
\begin{lemma}\label{le:metaSubs}
     Suppose that $X$ is a normed space and $(a_n)\subseteq X$ is metastable with rate $\Psi:(0,\infty)\times\NN^\NN\to\NN$. Then, for any nondecreasing $f:\NN\to\NN$ with $f(n)\ge n$, the sequence $(a_{f(n)})$ is metastable with rate $\tilde\Psi_f$ defined by $\tilde\Psi_f(\varepsilon,g):=\Psi(\varepsilon,g_f)$, where $g_f:\NN\to\NN$ is defined by $g_f(n):=f(n+g(n))-n$.
    \end{lemma}
    
    \begin{proof}
     Since $(a_n)$ is metastable with modulus $\Psi$,
     \begin{equation*}
      \forall\varepsilon>0\forall g:\NN\to\NN\exists n\le\Psi(\varepsilon,g_f)\forall i,j\in[n;f(n+g(n))]\big(\Vert a_i-a_j\Vert\le\varepsilon\big).
     \end{equation*}
     Since $f(n)\ge n$, we conclude
     \begin{equation*}
      \forall\varepsilon>0\forall g:\NN\to\NN\exists n\le\Psi(\varepsilon,g_f)\forall i,j\in[f(n);f(n+g(n))]\big(\Vert a_i-a_j\Vert\le\varepsilon\big).
     \end{equation*}
     The monotonicity of $f$ then implies
     \begin{equation*}
      \forall\varepsilon>0\forall g:\NN\to\NN\exists n\le\Psi(\varepsilon,g_f)\forall i,j\in[n;n+g(n)]\big(\Vert a_{f(i)}-a_{f(j)}\Vert\le\varepsilon\big),
     \end{equation*}
     so $\tilde\Psi_f$ is a rate of metastability for $(a_{f(n)})$.
    \end{proof}

    \begin{lemma}\label{le:metaStar}
     Suppose that $f:\NN\to\NN$ is strictly increasing and for each $k$, we have a statement $A(k)$. Define a function $(\cdot)^*:\{n\in\NN:n\ge f(0)\}\to\NN$ by $k\mapsto\max\{n\in\NN:f(n)\le k\}$. Then for all $g:\NN\to\NN$
     \begin{equation*}
      A(k)\textmd{ for all }k\in[n;n+\tilde g(n)]\Rightarrow A(k^*)\textmd{ for all }k\in[m;m+g(m)],
     \end{equation*}
     where $\tilde g(n):=\left(f(n)+g(f(n))\right)^*-n$ and $m:=f(n)$.

    \end{lemma}

    \begin{proof}
     Assume the statement $A(k)$ holds for all $k\in[n;n+\tilde g(n)]$. Observe that $n+\tilde g(n)=\bigl(m+g(m)\bigr)^*$ and $n=(f(n))^*=m^*$, so the statement $A(k)$ holds for all $k\in[m^*;(m+g(m))^*]$. Therefore,
     \begin{equation*}
      A(m^*)\wedge A(m^*+1)\wedge\ldots\wedge A((m+g(m))^*)
     \end{equation*}
    In particular, $(\cdot)^*$ is nondecreasing (since $f$ is nondecreasing) and so
     \begin{equation*}
      A(m^*)\wedge A((m+1)^*)\wedge\ldots\wedge A((m+g(m))^*).
     \end{equation*}
    Therefore, statement $A(k^*)$ holds for all $k\in[m;m+g(m)]$.
    \end{proof}

    We now give our main results, which were obtained by logical analysis of Bruck's proof \cite{Bruck(74)} using the proof-theoretic methods treated extensively in \cite{Kohlenbach(08)}.
    
        \begin{theorem}\label{PseudoMonotoneQuant}
      Let $C$ be a nonempty bounded closed convex subset of a Hilbert space $H$ with $\diam(C)\le M\in\NN$, $T:C\to C$ be a demicontinuous, single-valued pseudocontraction and $x_1,z\in C$. Suppose the sequences $(\lambda_n)$ and $(\theta_n)$ are acceptably paired with moduli as in Definition \ref{def:acceptablyQuant} satisfying $\lambda_n\left(1+\theta_n\right)\le1$, and the sequence $(y_i)$ defined by 
      \begin{equation*}
       y_i=\frac{1}{1+\theta_i}Ty_i+\frac{\theta_i}{1+\theta_i}z
      \end{equation*}
      is metastable with rate $\Psi:(0,\infty)\times\NN^\NN\to\NN$. Define the sequence $\left(x_n\right)$ by
      \begin{equation*}
       x_{n+1}=\left(1-\lambda_n\right)x_n+\lambda_nTx_n+\lambda_n\theta_n\left(z-x_n\right),
      \end{equation*}
      and a function $\Phi$ by $\Phi(\varepsilon,g,\varphi_1,\varphi_2,\varphi_3,\delta,n_0,M,f):=f(\tilde\Psi(\tilde\varepsilon,g_d)+n_1+d+1)$, where $\varphi_1,\varphi_2,\varphi_3:\RR\to\NN$, $f:\NN\to\NN$, $n_0\in\NN$ and $\delta>0$ are the moduli of Definition \ref{def:acceptablyQuant} and
      \begin{align*}
       &\tilde\Psi_f(\varepsilon,g):=\Psi(\varepsilon,g_f),&&g_f(n):=f(n+g(n))-n\\
       &g_d(n):=d+n_1+1+\tilde g(n+n_1+d+1)&&\tilde g(n):=\left(f(n)+g(f(n))\right)^*-n\\
       &d:=\max\{f(k_0),\lceil\log_c(\varepsilon/8M)\rceil\},&&k^*:=\max\{n\in\NN:f(n)\le k\}\\
       &n_1:=\max\{n_0,\varphi_1(\delta/2),\varphi_2(\tilde\varepsilon/4M^2),\varphi_3(\tilde\varepsilon^2/8M^2)\},&&k_0=\max\{\varphi_2(\varepsilon^2/6M^2),\varphi_3(\varepsilon^2/12M^2)\},\\
       &c:=\exp(-\delta/2)m,&&\tilde\varepsilon=\frac{1-c}{16}\cdot\varepsilon.
      \end{align*}
      To simplify notation, we will omit the dependence of $\Phi$ on the moduli for the parameters $(\lambda_n)$ and $(\theta_n)$ and instead write $\Phi(\varepsilon,g):=\Phi(\varepsilon,g,\varphi_1,\varphi_2,\varphi_3,\delta,n_0,M,f)$.
      Then,
      \begin{equation*}
       \forall\varepsilon>0\forall g:\NN\to\NN\exists n\le\Phi(g,\varepsilon)\forall i,j\in[n;n+g(n)]\left(\Vert x_i-y_{f(i^*)}\Vert\le\varepsilon\wedge\Vert y_{f(i^*)}-y_{f(j^*)}\Vert\le\varepsilon\right)
      \end{equation*}
    \end{theorem}
    
    \begin{remark}
     Observe that the bound given in Theorem \ref{PseudoMonotoneQuant} is independent of the operator $T$ and the space $H$. Moreover, it is also highly uniform with respect to the domain $C$ (dependence only via an upper bound on the diameter $\diam C$) and the choice of the parameter sequences $(\lambda_n)$ and $(\theta_n)$ (dependence only via the moduli $\varphi_1,\varphi_2,\varphi_3$, $\delta$, $n_0$ and $f$).
    \end{remark}

    \begin{proof}
      Since $T$ is pseudocontractive, $U:=I-T$ is monotone. Moreover,
      \begin{align*}
       y_i&=\frac{1}{1+\theta_i}Ty_i+\frac{\theta_i}{1+\theta_i}z\\
	  &=\frac{1}{1+\theta_i}(I-U)(y_i)+\frac{\theta_i}{1+\theta_i}z\\
	  &=\frac{1}{1+\theta_i}y_i-\frac{1}{1+\theta_i}Uy_i+\frac{\theta_i}{1+\theta_i}z,
      \end{align*}
    which is equivalent to 
    \begin{equation*}
     0=\left(1-\frac{1}{1+\theta_i}\right)y_i+\frac{1}{1+\theta_i}Uy_i-\frac{\theta_i}{1+\theta_i}z=\frac{\theta_i}{1+\theta_i}(y_i-z)+\frac{1}{1+\theta_i}Uy_i,
    \end{equation*}
    so $\theta_i(y_i-z)+Uy_i=0$. Moreover, the Bruck iteration rewritten in terms of $U$ reads
    \begin{align*}
     x_{n+1}&=(1-\lambda_n)x_n+\lambda_nTx_n+\lambda_n\theta_n(z-x_n)\\
	    &=x_n-\lambda_n\left(x_n-Tx_n+\theta_n(x_n-z)\right)\\
	    &=x_n-\lambda_n\left(Ux_n+\theta_n(x_n-z)\right).
    \end{align*}
    Therefore, for $n>i\ge2$,
    \begin{equation*}
     x_n-y_i=x_{n-1}-y_i-\lambda_{n-1}(Ux_{n-1}+\theta_{n-1}(x_{n-1}-z)),
    \end{equation*}
    so
    \begin{align}
     \Vert x_n-y_i\Vert^2&=\Big\langle x_{n-1}-y_i-\lambda_{n-1}(Ux_{n-1}+\theta_{n-1}(x_{n-1}-z)),\notag\\
			  &\qquad\qquad x_{n-1}-y_i-\lambda_{n-1}(Ux_{n-1}+\theta_{n-1}(x_{n-1}-z))\Big\rangle\notag\\
			  &=\Vert x_{n-1}-y_i\Vert^2-2\lambda_{n-1}\langle x_{n-1}-y_i,Ux_{n-1}+\theta_{n-1}(x_{n-1}-z)\rangle\notag\\
			  &\qquad+\lambda_{n-1}^2\Vert Ux_{n-1}+\theta_{n-1}(x_{n-1}-z)\Vert^2\notag\\
			  &=\Vert x_{n-1}-y_i\Vert^2+\lambda_{n-1}^2\Vert Ux_{n-1}+\theta_{n-1}(x_{n-1}-z)\Vert^2\notag\\
			  &\qquad-2\lambda_{n-1}\theta_{n-1}\langle x_{n-1}-y_i,x_{n-1}-z\rangle\notag\\
			  &\qquad-2\lambda_{n-1}\langle x_{n-1}-y_i,Ux_{n-1}\rangle\notag\\
			  &=\Vert x_{n-1}-y_i\Vert^2+\lambda_{n-1}^2\Vert Ux_{n-1}+\theta_{n-1}(x_{n-1}-z)\Vert^2\notag\\
			  &\qquad+2\lambda_{n-1}(\theta_i-\theta_{n-1})\langle x_{n-1}-y_i,x_{n-1}-z\rangle\label{eq:step}\\
			  &\qquad-2\lambda_{n-1}\langle x_{n-1}-y_i,Ux_{n-1}+\theta_i(x_{n-1}-z)\rangle.\notag
    \end{align}
    Since $U$ is monotone and $\theta_i(y_i-z)+Uy_i=0$,
    \begin{align*}
     \langle Ux_{n-1}+\theta_i(x_{n-1}-z),x_{n-1}-y_i\rangle&=\langle Ux_{n-1}+\theta_i(y_i-z),x_{n-1}-y_i\rangle+\theta_i\Vert x_{n-1}-y_i\Vert^2\\
							  &=\langle Ux_{n-1}-Uy_i,x_{n-1}-y_i\rangle+\theta_i\Vert x_{n-1}-y_i\Vert^2\\
							  &\ge\theta_i\Vert x_{n-1}-y_i\Vert^2.
    \end{align*}
    Equation \eqref{eq:step} then implies
    \begin{align*}
     \Vert x_n-y_i\Vert^2&\le(1-2\lambda_{n-1}\theta_i)\Vert x_{n-1}-y_i\Vert^2+\lambda_{n-1}^2\Vert Ux_{n-1}+\theta_{n-1}(x_{n-1}-z)\Vert^2\\
			  &\qquad+2\lambda_{n-1}(\theta_i-\theta_{n-1})\langle x_{n-1}-z,x_{n-1}-y_i\rangle.
    \end{align*}
    Observe that $\Vert Ux_{n-1}+\theta_{n-1}(x_{n-1}-z)\Vert=\Vert x_{n-1}-Tx_{n-1}+\theta_{n-1}(x_{n-1}-z)\Vert$. Since $\diam(C)\le M$, we conclude
    \begin{equation}\label{eq:ind1}
     \Vert x_n-y_i\Vert^2\le\exp(-2\lambda_{n-1}\theta_i)\Vert x_{n-1}-y_i\Vert^2+2M^2\lambda_{n-1}(\theta_i-\theta_{n-1})+4M^2\lambda_{n-1}^2.
    \end{equation}
    We show by induction on $n\ge i$ that
    \begin{equation}\label{eq:ind}
     \Vert x_n-y_i\Vert^2\le\exp\left(-2\theta_k\sum_{j=i}^{n-1}\lambda_j\right)\Vert x_i-y_i\Vert^2+2M^2\sum_{j=i}^{n-1}(\theta_i-\theta_j)\lambda_j+4M^2\sum_{j=i}^{n-1}\lambda_j^2.
    \end{equation}
    \textit{Proof of \eqref{eq:ind}}: For $n=i$ the inequality holds with equality. Suppose that the inequality holds true for some $n\ge i$. Then \eqref{eq:ind1} implies
    \begin{align*}
     \Vert x_{n+1}-y_i\Vert^2&\le\exp(-2\lambda_n\theta_i)\Vert x_n-y_i\Vert^2+2M^2\lambda_n(\theta_i-\theta_n)+4M^2\lambda_n^2\\
			      &\le\exp(-2\lambda_n\theta_i)\cdot\Bigg\{\exp\left(-2\theta_i\sum_{j=i}^{n-1}\lambda_j\right)\Vert x_i-y_i\Vert^2\\
			      &\qquad\qquad\qquad\qquad\qquad+2M^2\sum_{j=i}^{n-1}(\theta_i-\theta_j)\lambda_j+4M^2\sum_{j=i}^{n-1}\lambda_j^2\Bigg\}\\
			      &\qquad+2M^2\lambda_n(\theta_i-\theta_n)+4M^2\lambda_n^2\\
			      &=\exp\left(-2\theta_i\sum_{j=i}^{n}\lambda_j\right)\Vert x_i-y_i\Vert^2+2M^2\sum_{j=i}^{n}(\theta_i-\theta_j)\lambda_j+4M^2\sum_{j=i}^{n}\lambda_j^2,
    \end{align*}
    which is what we needed to show.
    
    Since $\theta_i-\theta_j\le\theta_i-\theta_n$ for $i\le j\le n$, \eqref{eq:ind} implies
    \begin{equation}\label{eq:ind2}
     \Vert x_n-y_i\Vert^2\le\exp\left(-2\theta_i\sum_{j=i}^{n-1}\lambda_j\right)\Vert x_i-y_i\Vert^2+2M^2(\theta_i-\theta_n)\sum_{j=i}^{n-1}\lambda_j+4M^2\sum_{j=i}^{n-1}\lambda_j^2,\textmd{ for all }n\ge i.
    \end{equation}
    Now let $f(n)$ be the subsequence of Definition \ref{def:acceptably}. We now prove that $(x_{f(n)})$ is Cauchy. Taking 
    $i=f(k)$ and $n=f(k+1)$ in \eqref{eq:ind2}, we get
    \begin{align}\label{eq:crucSubs}
     \Vert x_{f(k+1)}-y_{f(k)}\Vert^2&\le\exp\left(-2\theta_{f(k)}\sum_{j=f(k)}^{f(k+1)}\lambda_j\right)\cdot\exp(2\theta_{f(k)}\lambda_{f(k+1)})\cdot\Vert x_{f(k)}-y_{f(k)}\Vert^2\notag\\
		    &\qquad+2M^2(\theta_{f(k)}-\theta_{f(k+1)})\cdot\sum_{j=f(k)}^{f(k+1)}\lambda_j+4M^2\sum_{j=f(k)}^{f(k+1)}\lambda_j^2.
    \end{align}
    Now observe that
    \begin{equation*}
     \exp\left(-2\theta_{f(k)}\sum_{j=f(k)}^{f(k+1)}\lambda_j\right)\le\exp(-2\delta)<1,\quad\textmd{for all }k\ge n_0.
    \end{equation*}
    Moreover, $(\theta_n)$ is a null sequence with modulus $\varphi_1$. Thus, $\exp(2\theta_{f(k)}\lambda_{f(k+1)})\le\exp(2\theta_k)\le\exp(\delta)$ for all $k\ge\varphi_1(\delta/2)$. Furthermore, for all $k\ge\max\{\varphi_2(\tilde\varepsilon^2/4M^2),\varphi_3(\tilde\varepsilon^2/8M^2)\}$, the remainder term in \eqref{eq:crucSubs} is less than $\tilde\varepsilon^2$. In total, 
    \begin{equation*}
     \Vert x_{f(k+1)}-y_{f(k)}\Vert^2\le\exp(-\delta)\cdot\Vert x_{f(k)}-y_{f(k)}\Vert^2+\tilde\varepsilon^2,\quad\textmd{for all }k\ge n_1
    \end{equation*}
    since $n_1=\max\{n_0,\varphi_1(\delta/2),\varphi_2(\tilde\varepsilon^2/4M^2),\varphi_3(\tilde\varepsilon^2/2M^2)\}$. Because $c=\exp(-\delta/2)$, we then get
    \begin{align}\label{eq:mono}
     \Vert x_{f(k+1)}-y_{f(k)}\Vert&\le c\cdot\Vert x_{f(k)}-y_{f(k)}\Vert+\tilde\varepsilon\notag\\
				   &\le c\cdot\Vert x_{f(k)}-y_{f(k-1)}\Vert+c\cdot\Vert y_{f(k-1)}-y_{f(k)}\Vert+\tilde\varepsilon.
    \end{align}

    Now observe that since $(y_n)$ is metastable with rate $\Psi$, the subsequence $(y_{f(n)})$ is metastable with rate $\tilde\Psi$ by Lemma \ref{le:metaSubs}. Thus, there exists an integer $n\le\tilde\Psi(\tilde\varepsilon,g_d)$ such that $\Vert y_{f(k)}-y_{f(j)}\Vert\le\tilde\varepsilon$ for all $k,j\in[n;n+g_d(n)]$. Taking $n_2:=n+n_1$, we have on the one hand $n_2\ge n_1$, and $\Vert y_{f(k)}-y_{f(j)}\Vert\le\tilde\varepsilon$ for all $k,j\in[n_2;n_2+d+1+\tilde g(n_2+d+1)]$ on the other. Setting $j=k-1$, we conclude
    \begin{equation}\label{eq:mono1}
     \Vert y_{f(k)}-y_{f(k-1)}\Vert\le\tilde\varepsilon,\quad\textmd{for all }k\in[n_2+1;n_2+d+1+\tilde g(n_2+d+1)].
    \end{equation}
    Suppose now that $k\in[n_2+d;n_2+d+\tilde g(n_2+d+1)]$. Then \eqref{eq:mono1} and \eqref{eq:mono} yield
    \begin{align*}
     \Vert x_{f(k+1)}-y_{f(k)}\Vert&\le c\cdot\Vert x_{f(k)}-y_{f(k-1)}\Vert+2\tilde\varepsilon\\
				   &\le c\cdot(c\cdot\Vert x_{f(k-1)}-y_{f(k-2)}\Vert+2\tilde\varepsilon)+2\tilde\varepsilon\\
				   &=c^2\cdot\Vert x_{f(k-1)}-y_{f(k-2)}\Vert+2\tilde\varepsilon\cdot c+2\tilde\varepsilon\\
				   &\le\ldots\\
				   &=c^{d-1}\cdot\Vert x_{f(k-d+2)}-y_{f(k-d+1)}\Vert+2\tilde\varepsilon\sum_{k=0}^{d-2}c^k\\
				   &\le c^d\cdot\Vert x_{f(k-d+1)}-y_{f(k-d)}\Vert+c^{d-1}\cdot\Vert y_{f(k-d+1)}-y_{f(k-d)}\Vert+2\tilde\varepsilon\sum_{k=0}^{d-2}c^k\\
				   &\le c^d\cdot M+2\tilde\varepsilon\sum_{k=0}^{d-1}c^k\\
				   &\le c^d\cdot M+2\tilde\varepsilon\sum_{k=0}^{\infty}c^k\\
				   &=c^d\cdot M+\frac{2\tilde\varepsilon}{1-c}.
    \end{align*}
    Since $d\ge\log_c(\varepsilon/8M)$ and $\tilde\varepsilon=\frac{1-c}{16}\cdot\varepsilon$, we have
    \begin{equation*}
     \Vert x_{f(k+1)}-y_{f(k)}\Vert\le\frac{\varepsilon}{4},\textmd{ for all } k\in[n_2+d;n_2+d+\tilde g(n_2+d+1)].
    \end{equation*}
    Therefore, setting $n_3:=n_2+d+1$ and using \eqref{eq:mono1},
    \begin{align*}
     \Vert x_{f(k)}-y_{f(k)}\Vert&\le\Vert x_{f(k)}-y_{f(k-1)}\Vert+\Vert y_{f(k)}-y_{f(k-1)}\Vert\le\frac{\varepsilon}{4}+\tilde\varepsilon\\
     &<\frac{\varepsilon}{3},\quad\textmd{for all } k\in[n_3;n_3+\tilde g(n_3)].
    \end{align*}
    By Lemma \ref{le:metaStar}
    \begin{equation}\label{eq:metaDist}
     \Vert x_{f(k^*)}-y_{f(k^*)}\Vert\le\varepsilon/3,\textmd{ for all }k\in[f(n_3);f(n_3)+g(f(n_3))]
    \end{equation}

    Now, for $k\ge f(0)$, observe that $k^*$ denotes the unique integer such that $f(k^*)\le k<f(k^*+1)$. Take $n=k$, $i=f(k^*)$ in \eqref{eq:ind2}; since the exponential factor is less than or equal to $1$, 
    \begin{align*}
     \Vert x_k-y_{f(k^*)}\Vert^2&\le\Vert x_{f(k^*)}-y_{f(k^*)}\Vert^2+2M^2(\theta_{f(k^*)}-\theta_k)\sum_{j=f(k^*)}^{k-1}\lambda_j+4M^2\sum_{j=f(k^*)}^{k-1}\lambda_j^2\\
			     &\le\Vert x_{f(k^*)}-y_{f(k^*)}\Vert^2+2M^2(\theta_{f(k^*)}-\theta_{f(k^*+1)})\sum_{j=f(k^*)}^{f(k^*+1)}\lambda_j+4M^2\sum_{j=f(k^*)}^{f(k^*+1)}\lambda_j^2.
    \end{align*}
    Observe that the latter two terms become less than $\varepsilon^2/3$ whenever $k^*\ge k_0$ since, by definition, $k_0=\max\{\varphi_2(\varepsilon^2/6M^2),\varphi_3(\varepsilon^2/12M^2)\}$. But this is always the case whenever $k\ge f(k_0)$ since then $k^*\ge (f(k_0))^*=k_0$ by the monotonicity of $(\cdot)^*$. Therefore,
    \begin{equation}\label{eq:bla}
     \Vert x_k-y_{f(k^*)}\Vert^2\le\Vert x_{f(k^*)}-y_{f(k^*)}\Vert^2+\frac{2\varepsilon^2}{3},\textmd{ for all }k\ge f(k_0).
    \end{equation}
   Since $f(n_3)=f(n_2+d+1)\ge f(d)\ge d\ge f(k_0)$, equations \eqref{eq:metaDist}, \eqref{eq:bla} together imply
    \begin{align}\label{eq:bla1}
     \Vert x_k-y_{f(k^*)}\Vert&\le\varepsilon,\quad\textmd{for all }k\in[f(n_3);f(n_3)+g(f(n_3))].
    \end{align}
    
    Now recall that $\Vert y_{f(i)}-y_{f(j)}\Vert\le\tilde\varepsilon$ for all $i,j\in[n_2;n_2+d+1+\tilde g(n_2+d+1)]$. 
    Again by Lemma \ref{le:metaStar}, this implies
    \begin{equation}\label{eq:bla2}
     \Vert y_{f(i^*)}-y_{f(j^*)}\Vert\le\tilde\varepsilon\le\varepsilon,\textmd{ for all }i,j\in[f(n_3),f(n_3)+g(f(n_3))].
    \end{equation}
    Therefore, $f(n_3)=f(n_2+d+1)\le f(\tilde\Psi(g_d,\tilde\varepsilon)+n_1+d+1)$ satisfies the claim.
    \end{proof}

  \begin{theorem}\label{thm:pseudoQuant2}
    In the situation of Theorem \ref{PseudoMonotoneQuant}, $(x_n)$ is metastable with rate $\Phi'(\varepsilon,g):=\Phi(\varepsilon/3,g)$.
  \end{theorem}
  
  \begin{proof}
    Since $\tilde\varepsilon<\varepsilon$ and $\Vert x_i-x_j\Vert\le\Vert x_i-y_{f(i^*)}\Vert+\Vert x_j-y_{f(j^*)}\Vert+\Vert y_{f(i^*)}-y_{f(j^*)}\Vert$, equations \eqref{eq:bla1} and \eqref{eq:bla2} imply
    \begin{equation*}
      \forall\varepsilon>0\forall g:\NN\to\NN\exists n\le\Phi(\varepsilon/3,g)\forall i,j\in[n;n+g(n)]\left(\Vert x_i-x_j\Vert<\varepsilon\right),
     \end{equation*}
     which is what we needed to show.
  \end{proof}

   \begin{theorem}\label{thm:deliver} In the situation of Theorem \ref{PseudoMonotoneQuant},
   \begin{multline*}
    \forall\varepsilon>0\,\forall g:\NN\to\NN\,\exists n\le\Phi''(\varepsilon,g)\,\forall i,j\in[n;n+g(n)]\\\big(\Vert x_i-x_j\Vert<\varepsilon\,\wedge\Vert x_i-y_i\Vert<\varepsilon\,\wedge\,\Vert y_i-Ty_i\Vert<\varepsilon\big),
   \end{multline*}
  where $\Phi''(\varepsilon,g):=\hat\Phi(\varepsilon/3,)$, and $\hat\Phi$ is defined like $\Phi$ in Theorem \ref{PseudoMonotoneQuant}, but with $\hat g_d$, defined by
  \begin{equation*}
   \hat g_d(n):=d+\hat n_1+1+\hat g(n+d+\hat n_1+1).
  \end{equation*}
  instead of $g_d$, where $\hat g(n):=(f(n)+g(f(n)))^*-n+1$ and $\hat n_1:=\max\{n_1,\varphi_1(\varepsilon/M)\}$.
  
  Moreover, we can take $\Psi(\varepsilon,g):=\tilde g^{(\lceil 16d^2/\varepsilon^2\rceil)}(1)$, where $\tilde g(n)=n+1+g(n+1)$.
  \end{theorem}
  
  \begin{proof}
   By altering the definition of $g_d$, the point $f(n_3)$ that satisfies the conclusion of Theorem \ref{thm:pseudoQuant2} also satisfies $f(n_3)\ge n_3=n_2+d+1=n_0+\hat n_1+d+1\ge\varphi_1(\varepsilon,g)$. Therefore, $(x_n)$ is metastable, and
   \begin{equation*}
    \Vert y_i-Ty_i\Vert=\frac{\theta_i}{1+\theta_i}\Vert Ty_i-z\Vert\le\theta_i\cdot M\le\varepsilon,\text{ for all }i\ge n.
   \end{equation*}
   It remains to verify that $\Vert x_i-y_i\Vert\le\varepsilon$ on $[f(n_3);f(n_3)+g(f(n_3))]$. To this end, observe that
   \begin{equation}\label{eq:conf}
    \Vert y_i-y_j\Vert\le\varepsilon/3,\text{ for all }i,j\in[f(n_3);f(n_3+\hat g(n_3))].
   \end{equation}
   Since $f(k^*+1)>k\ge f(k^*)$ for all $k\ge f(0)$, we conclude $f(n_3+\hat g(n_3))=f((f(n_3)+g(f(n_3)))^*+1)\ge f(n_3)+g(f(n_3))$. Moreover, $f((f(n_3))^*)=f(n_3)$. Therefore, \eqref{eq:conf} implies
   \begin{equation*}
    \Vert y_k-y_{f(k^*)}\Vert\le\varepsilon/3,\text{ for all }k\in[f(n_3),f(n_3)+g(f(n_3))],
   \end{equation*}
   which, using \eqref{eq:bla1}, implies for all $k\in[f(n_3),f(n_3)+g(f(n_3))]$ 
   \begin{equation*}
    \Vert x_k-y_k\Vert\le\Vert x_k-y_{f(k^*)}\Vert+\Vert y_{f(k^*)}-y_k\Vert\le\varepsilon/3+\varepsilon/3<\varepsilon.
   \end{equation*}
   That we may choose $\Psi(\varepsilon,g):=\tilde g^{(\lceil 16d^2/\varepsilon^2\rceil)}(1)$ follows from Theorem 2.8 and Corollary 2.9 of \cite{KoernleinKohlenbach(13)}.
  \end{proof}

  \begin{remark}
   Observe that Theorems \ref{PseudoMonotoneQuant}, \ref{thm:pseudoQuant2} and \ref{thm:deliver} require only the demicontinuity of $T$. Therefore, model-theoretic approaches (cf.~\cite{HensonIovino(02)}) are not applicable, as these always require norm-continuity.
  \end{remark}
  
  \begin{remark}
   Suppose $\Psi$ does not depend on $g$ for a concrete choice of the input. Then metastability for $(y_n)$ would read
   \begin{equation*}
     \forall\varepsilon>0\,\forall g:\NN\to\NN\,\exists n\le\Psi(\varepsilon)\,\forall i,j\in[n;n+g(n)]\,\big(\Vert x_i-x_j\Vert<\varepsilon\big).
   \end{equation*}
   This is logically equivalent to
   \begin{equation*}
    \forall\varepsilon>0\,\forall g:\NN\to\NN\,\exists n\le\Psi(\varepsilon)\,\forall i,j\ge n\,\big(\Vert x_i-x_j\Vert<\varepsilon\big),
   \end{equation*}
   i.e.~a rate of convergence. In this case, we would get in Theorem \ref{PseudoMonotoneQuant} a rate of convergence $\Phi(\varepsilon):=f(\Psi(\tilde\varepsilon)+n_1+d+1)$, where, as before
         \begin{align*}
       &n_1:=\max\{n_0,\varphi_1(\delta/2),\varphi_2(\tilde\varepsilon/4M^2),\varphi_3(\tilde\varepsilon^2/8M^2)\},&&k_0=\max\{\varphi_2(\varepsilon^2/6M^2),\varphi_3(\varepsilon^2/12M^2)\},\\
       &d:=\max\{f(k_0),\lceil\log_c(\varepsilon/8M)\rceil\},&&\tilde\varepsilon=\frac{1-c}{16}\cdot\varepsilon,\\
       &c:=\exp(-\delta/2).&&
      \end{align*}
  \end{remark}

        \begin{theorem}\label{thm:unifcont}
     Suppose that in the situation of Theorem \ref{PseudoMonotoneQuant}, $T$ is additionally uniformly continuous on $C$ with modulus $\omega$. For $g:\NN\to\NN$ define $g_{b}(n):=b+g(n+b)$, where $b:=f((\varphi_1(\varepsilon/3M))^*+1)$. Then
     \begin{equation*}
       \forall\varepsilon>0\forall g:\NN\to\NN\exists n\le\Delta(g,\varepsilon)\forall i,j\in[n;n+g(n)]\bigl(\Vert x_i-x_j\Vert\le\varepsilon\wedge\Vert x_i-Tx_i\Vert\le\varepsilon\bigr),
      \end{equation*}
      where $\Delta(g,\varepsilon):=\Phi(g_b,\min\{\varepsilon/3,\omega(\varepsilon/3M)\})+b$.
    \end{theorem}
    \begin{proof}
     By Theorem \ref{PseudoMonotoneQuant}, there exists a $k\le\Phi(g_b,\min\{\varepsilon/3,\omega(\varepsilon/3M)\})$ such that for $n:=k+b$,
     \begin{equation*}
      \Vert x_i-y_{f(i^*)}\Vert\le\min\{\varepsilon/3,\omega(\varepsilon/3M)\},\textmd{ for all }i,j\in[n;n+g(n)].
     \end{equation*}
    Now observe that $b^*=(\varphi_1(\varepsilon/3M))^*+1$ since $(f(k))^*=k$ for all nonnegative integers $k$. Therefore,
    $f(n^*)\ge f(b^*)=f((\varphi_1(\varepsilon/3M))^*+1)>\varphi_1(\varepsilon/3M)$. Thus, $\theta_{f(i^*)}\le\varepsilon/3M$ for all $i\ge n$. Consequently
    \begin{equation*}
     \Vert y_{f(i^*)}-Ty_{f(i^*)}\Vert=\frac{\theta_{f(i^*)}}{1+\theta_{f(i^*)}}\Vert Ty_{f(i^*)}-z\Vert\le\theta_{f(i^*)}\cdot M\le\frac{\varepsilon}3,\textmd{ for all }i\ge n.
    \end{equation*}
    Therefore,
    \begin{equation*}
     \Vert x_i-Tx_i\Vert\le\Vert x_i-y_{f(i^*)}\Vert+\Vert  y_{f(i^*)}-Ty_{f(i^*)}\Vert+\Vert Tx_i-Ty_{f(i^*)}\Vert\le\varepsilon,\textmd{ for all }i\in[n;n+g(n)].	
    \end{equation*}
    That $\Vert x_i-x_j\Vert\le\varepsilon$ on $[n;n+g(n)]$ follows as in Corollary \ref{thm:pseudoQuant2}.
    \end{proof}

    \section{Application to Concrete Instances}\label{sec:ex}
    
    In this section, we compute explicitly the moduli $\varphi_1,\varphi_2,\varphi_3,n_0$ and $\delta$ for the two examples of parameter sequences of Remark \ref{rem:ex}. We then compare the bound to the one obtained in \cite{Kohlenbach(11)} for Halpern iterations of nonexpansive mappings.
    
    \subsection{Example 1}\label{sec:ex1}
    Suppose $p$ and $q$ are real numbers in $(0,1)$ such that $0<q<\min\{p,1-p\}$, and take $\lambda_n:=n^{-p}$ and $\theta_n:=n^{-q}$. Set $r=(p+q)/2$. There are two cases to consider, namely $p\ge1/2$ and $p<1/2$.
    
    If $p<1/2$, then $1/2>p>r>q$, so we conclude $1/2<1-p<1-r<1-q$, whence $0<1-2p<1-r-p<1-p-q$. Then,
    \begin{equation*}
     \frac{1-2p}{1-p}<\frac{1-r-p}{1-p}<\frac{1-p-q}{1-p}=1-\frac{q}{1-p},
    \end{equation*}
    and so
    \begin{equation*}
     \frac{1-p}{1-2p}>\frac{1-p}{1-r-p}>\left(1-\frac{q}{1-p}\right)^{-1}.
    \end{equation*}
    Thus, if we choose $d:=\min\left\{\frac{1-p}{1-r-p},\frac{3}{2}\left(1-\frac{q}{1-p}\right)^{-1}\right\}$, then 
    \begin{equation*}
     \left(1-\frac{q}{1-p}\right)^{-1}<d<\min\left\{\frac{1-p}{1-2p},2\left(1-\frac{q}{1-p}\right)^{-1}\right\}\quad \text{(for $p< 1/2$)},
    \end{equation*}
    which is Bruck's condition. For $p\ge\frac{1}{2}$, we see as before that
    \begin{equation*}
     \left(1-\frac{q}{1-p}\right)^{-1}<d<2\left(1-\frac{q}{1-p}\right)^{-1}\quad\text{(for $p\ge1/2$)}.
    \end{equation*}
    An important consequence of our choice of $d$ is that $d>1$, which we will use throughout this section.
    
    Now, one can take $f(n):=\lceil n^{d/(1-p)}\rceil$. To calculate the other moduli, we need the following Lemma, which is a direct consequence of Taylor's Theorem using the Lagrange remainder term.
    \begin{lemma}\label{lem:taylor}
     Suppose $x,r\in\RR$ with $x>0$ and $r\ne 1$. Then,
     \begin{enumerate}[label=(\roman*)]
      \item there exists a real number $\xi\in(x,x+1)$ such that $(x+1)^r=x^r+r\xi^{r-1}$, and
      \item there exists a real number $\nu\in(x-1,x)$ such that $(x-1)^r=x^r-r\nu^{r-1}$.
     \end{enumerate}

    \end{lemma}
    We now proceed to calculate the moduli. Observe that
    \begin{align}\label{eq:nat}
     \theta_{f(n)}\cdot\sum_{j=f(n)}^{f(n+1)}\lambda_j&=\left\lceil n^{d/(1-p)}\right\rceil^{-q}\cdot\sum_{j=f(n)}^{f(n+1)}j^{-p}\notag\\
	      &>(n^{d/(1-p)}+1)^{-q}\cdot\int_{f(n)}^{f(n+1)}j^{-p}dj\notag\\
	      &=\frac{(n^{d/(1-p)}+1)^{-q}}{1-p}\left[(f(n+1))^{1-p}-(f(n))^{1-p}\right]\notag\\
	      &=\frac{(n^{d/(1-p)}+1)^{-q}}{1-p}\left(\left\lceil(n+1)^{d/(1-p)}\right\rceil^{1-p}-\left\lceil n^{d/(1-p)}\right\rceil^{1-p}\right)\notag\\
	      &\ge\frac{(n^{d/(1-p)}+1)^{-q}}{1-p}\left[\left((n+1)^{d/(1-p)}\right)^{1-p}-\left(n^{d/(1-p)}+1\right)^{1-p}\right]\notag\\
	      &=\frac{(n^{d/(1-p)}+1)^{-q}}{1-p}\left[(n+1)^d-\left(n^{d/(1-p)}+1\right)^{1-p}\right]
    \end{align}
    By virtue of \ref{lem:taylor}, there exists a $\xi\in(n^{d/(1-p)},n^{d/(1-p)}+1)$ such that
    \begin{align*}
     \left( n^{d/(1-p)}+1\right)^{1-p}&=n^d+(1-p)\xi^{-p}\\
				    &\le n^d+(1-p)n^{-\frac{dp}{1-p}}.
    \end{align*}
    Therefore, applying Lemma \ref{lem:taylor}, there exists a $\xi\in(n,n+1)$ such that
    \begin{align*}
    (n+1)^d-\left(n^{d/(1-p)}+1\right)^{1-p}&\ge(n+1)^d-n^d-(1-p)n^{-\frac{dp}{1-p}}\\
			&=n^d+d\xi^{d-1}-n^d-(1-p)n^{-\frac{dp}{1-p}}\\
			&\ge n^d+dn^{d-1}-n^d-(1-p)n^{-\frac{dp}{1-p}}\\
			&=dn^{d-1}-(1-p)n^{-\frac{dp}{1-p}}.
    \end{align*}
    Consequently, going back to \eqref{eq:nat},
    \begin{equation*}
     \theta_{f(n)}\cdot\sum_{j=f(n)}^{f(n+1)}\lambda_j>\frac{(n^{d/(1-p)}+1)^{-q}}{1-p}\left(dn^{d-1}-(1-p)n^{-\frac{dp}{1-p}}\right).
     \end{equation*}
     By Lemma \ref{lem:taylor}, there now exists $\xi\in(n^{d/(1-p)},n^{d/(1-p)}+1)$ such that 
     \begin{align*}
      \theta_{f(n)}\cdot\sum_{j=f(n)}^{f(n+1)}\lambda_j&>\frac{n^{\frac{-dq}{1-p}}-q\xi^{-q-1}}{1-p}\left(dn^{d-1}-(1-p)n^{-\frac{dp}{1-p}}\right)\\
      &\ge\frac{n^{\frac{-dq}{1-p}}-qn^{-\frac{d(q+1)}{1-p}}}{1-p}\left(dn^{d-1}-(1-p)n^{-\frac{dp}{1-p}}\right)\\
      &\ge\frac{d}{1-p}n^{d-1-\frac{dq}{1-p}}-n^{-\frac{d(p+q)}{1-p}}-\frac{dq}{1-p}n^{d-1-\frac{d(q+1)}{1-p}}.
     \end{align*}
    Now observe that $d-1-\frac{dq}{1-p}=d(1-\frac{q}{1-p})-1>\frac{d}{d}-1=0$ and $d-1-\frac{d(q+1)}{1-p}=d(1-\frac{q}{1-p})-\frac{d}{1-p}-1\le3/2-2=-1/2$. Moreover, $-\frac{d(p+q)}{1-p}<0$, so the right-hand-side in the equation above is monotone increasing. Therefore,
    \begin{equation*}
     \theta_{f(n)}\cdot\sum_{j=f(n)}^{f(n+1)}\lambda_j>\frac{d}{1-p}-1-\frac{dq}{1-p}=\frac{1-q}{1-p}\cdot d-1,\quad\textmd{for all }n\ge1.
    \end{equation*}
    Since $q<p$, we have $1-q>1-p$. Moreover, $d>1$. Thus, we may choose $n_0:=1$ and $\delta:=\frac{d(1-q)}{1-p}-1>0$.
    
    We now calculate the modulus $\varphi_2$. By Lemma \ref{lem:taylor}, there exists a real number\\$\xi\in\left((n+1)^{d/(1-p)},(n+1)^{d/(1-p)}+1\right)$ such that
    \begin{align*}
     \theta_{f(n)}-\theta_{f(n+1)}&=\left\lceil n^{d/(1-p)}\right\rceil^{-q}-\left\lceil(n+1)^{d/(1-p)}\right\rceil^{-q}\\
				  &\le n^{-\frac{dq}{1-p}}-\left((n+1)^{d/(1-p)}+1\right)^{-q}\\
				  &=n^{-\frac{dq}{1-p}}-\left((n+1)^{-\frac{dq}{1-p}}-q\cdot\xi^{-q-1}\right)\\
				  &\le n^{-\frac{dq}{1-p}}-(n+1)^{-\frac{dq}{1-p}}+q\cdot\left((n+1)^{d/(1-p)}\right)^{-q-1}\\
				  &=n^{-\frac{dq}{1-p}}-(n+1)^{-\frac{dq}{1-p}}+q\cdot(n+1)^{-\frac{d(1+q)}{1-p}}\\
				  &\le n^{-\frac{dq}{1-p}}-(n+1)^{-\frac{dq}{1-p}}+q\cdot n^{-\frac{d(1+q)}{1-p}}.
    \end{align*}
    Applying once more Lemma \ref{lem:taylor}, we see that for some $\xi\in(n,n+1)$,
    \begin{align*}
     \theta_{f(n)}-\theta_{f(n+1)}&=n^{-\frac{dq}{1-p}}-\left(n^{-\frac{dq}{1-p}}-\frac{dq}{1-p}\xi^{-\frac{dq}{1-p}-1}\right)+q\cdot n^{-\frac{d(1+q)}{1-p}}\\
			&=\frac{dq}{1-p}\xi^{-\frac{dq}{1-p}-1}+q\cdot n^{-\frac{d(1+q)}{1-p}}\\
			&\le\frac{dq}{1-p}n^{-\frac{dq}{1-p}-1}+q\cdot n^{-\frac{d(1+q)}{1-p}}.
    \end{align*}
    Since $-\frac{d(1+q)}{1-p}=-\frac{dq}{1-p}-\frac{d}{1-p}\le-\frac{dq}{1-p}-1$,
    \begin{equation}
     \theta_{f(n)}-\theta_{f(n+1)}\le\left(\frac{dq}{1-p}+q\right)n^{-\frac{dq}{1-p}-1}=\frac{q(d+1-p)}{1-p}n^{-\frac{dq}{1-p}-1}\label{eq:nat2}.
    \end{equation}

    On the other hand,
    \begin{align*}
     \sum_{j=f(n)}^{f(n+1)}\lambda_j&=\sum_{j=f(n)}^{f(n+1)}j^{-p}\le\int_{f(n)-1}^{f(n+1)-1}j^{-p}dj\\
      &=\frac1{1-p}\left(\left\lceil(n+1)^{d/(1-p)}-1\right\rceil^{1-p}-\left\lceil n^{d/(1-p)}-1\right\rceil^{1-p}\right).
    \end{align*}
    Since $1-p>0$, we conclude
    \begin{equation*}
     \sum_{j=f(n)}^{f(n+1)}\lambda_j\le\frac1{1-p}\left((n+1)^d-\left(n^{d/(1-p)}-1\right)^{1-p}\right).
    \end{equation*}
    By Lemma \ref{lem:taylor}, there exists $\xi\in(n,n+1)$ and $\nu\in(n^{d/(1-p)}-1,n^{d/(1-p)})$ such that for $n\ge2$
    \begin{align*}
     (1-p)\cdot\sum_{j=f(n)}^{f(n+1)}\lambda_j&\le n^d+d\xi^{d-1}-\left(n^{d/(1-p)}-1\right)^{1-p}\\
				&\le n^d+d(n+1)^{d-1}-\left(n^{d/(1-p)}-1\right)^{1-p}\\
				&=n^d+d(n+1)^{d-1}-n^d+(1-p)\nu^{-p}\\
				&\le n^d+d(n+1)^{d-1}-n^d+(1-p)\left(n^{d/(1-p)}-1\right)^{-p}\\
				&\le d(n+1)^{d-1}+(1-p)\cdot2^pn^{-\frac{dp}{1-p}}\\
				&\le d\cdot2^{d-1}n^{d-1}+(1-p)\cdot2^pn^{-\frac{dp}{1-p}}.
    \end{align*}
    Combining this with \eqref{eq:nat2}, we get
    \begin{align*}
     \left(\theta_{f(n)}-\theta_{f(n+1)}\right)\sum_{j=f(n)}^{f(n+1)}\lambda_j&\le\left(\frac{q(d+1-p)}{1-p}n^{-\frac{dq}{1-p}-1}\right)\frac{d\cdot2^{d-1}n^{d-1}+(1-p)\cdot2^pn^{-\frac{dp}{1-p}}}{1-p}\\
     &=\frac{2^{d-1}dq(d+1-p)}{(1-p)^2}n^{d-2-\frac{dq}{1-p}}+\frac{2^pq(d+1-p)}{1-p}n^{-\frac{d(p+q)}{1-p}-1}.
    \end{align*}
    Now observe that $d-2-\frac{dq}{1-p}=d(1-\frac{q}{1-p})-2\le\frac{3}{2}(1-\frac{q}{1-p})^{-1}(1-\frac{q}{1-p})-2=-\frac{1}{2}$. Therefore,
    \begin{align*}
     \left(\theta_{f(n)}-\theta_{f(n+1)}\right)\sum_{j=f(n)}^{f(n+1)}\lambda_j&\le\frac{2^{d-1}dq(d+1-p)}{(1-p)^2}n^{-\frac{1}{2}}+\frac{2^pq(d+1-p)}{1-p}n^{-\frac{d(p+q)}{1-p}-1}\\
     &\le\frac{2^{d-1}dq(d+1-p)}{(1-p)^2}n^{-\frac{1}{2}}+\frac{2^pq(d+1-p)}{(1-p)^2}n^{-\frac{1}{2}}\\
     &\le\frac{2^ddq(d+1-p)}{(1-p)^2}n^{-\frac{1}{2}}.
    \end{align*}
    Therefore, we may take $\varphi_2(\varepsilon):=\left\lceil\left(\frac{2^{d}dq(d+1-p)}{\varepsilon(1-p)^2}\right)^2\right\rceil$.
    
    Now, we calculate the modulus $\varphi_3$. To this end, observe that
    \begin{multline*}
     \sum_{j=f(n)}^{f(n+1)}\lambda_j^2=\sum_{j=f(n)}^{f(n+1)}j^{-2p}\le\int_{j=f(n)-1}^{f(n+1)-1}j^{-2p}dj\\=\left\{
     \begin{array}{ll}
         \frac1{1-2p}\left(\left\lceil(n+1)^{d/(1-p)}-1\right\rceil^{1-2p}-\left\lceil n^{d/(1-p)}-1\right\rceil^{1-2p}\right)&\textmd{, if }p\ne{\frac12},\\
         \log\left\lceil(n+1)^{d/(1-p)}-1\right\rceil-\log\left\lceil n^{d/(1-p)}-1\right\rceil&\textmd{, if }p=\frac12.
     \end{array}\right.
    \end{multline*}
    We have to distinguish the cases $p>1/2$, $p=1/2$ and $p<1/2$. For $p>1/2$, we use the estimate
    \begin{equation*}
     \sum_{j=f(n)}^{f(n+1)}\lambda_j^2\le\frac{\left(n^{d/(1-p)}-1\right)^{1-2p}}{2p-1}\le\varepsilon,\quad\text{for all }n\ge\left(\bigl((2p-1)\varepsilon\bigr)^{1/(1-2p)}+1\right)^{(1-p)/d}.
    \end{equation*}
    For $p=1/2$, we see that
    \begin{equation*}
     \sum_{j=f(n)}^{f(n+1)}\lambda_j^2\le\log\frac{(n+1)^{d/(1-p)}}{\left\lceil n^{d/(1-p)}-1\right\rceil}\le\log\frac{(n+1)^{d/(1-p)}}{n^{d/(1-p)}-1}.
    \end{equation*}
    Now observe that $(n+1)^{d/(1-p)}=n^{d/(1-p)}+\frac{d}{1-p}\xi^{\frac{d}{1-p}-1}$ for some $\xi\in(n,n+1)$. Therefore,
    \begin{align*}
     \sum_{j=f(n)}^{f(n+1)}\lambda_j^2&\le\log\left(\frac{n^{d/(1-p)}+\frac{d}{1-p}(n+1)^{\frac{d}{1-p}-1}}{n^{d/(1-p)}-1}\right)\\
     &=\log\left(1+\frac{\frac{d}{1-p}(n+1)^{\frac{d}{1-p}-1}+1}{n^{d/(1-p)}-1}\right)\\
     &\le\log\left(1+2^{d/(1-p)}\cdot2^{d/(1-p)}\cdot\frac{\frac{d}{1-p}n^{\frac{d}{1-p}-1}}{n^{d/(1-p)}}\right)\\
     &=\log\left(1+\frac{d\cdot2^{\frac{2d}{1-p}}}{1-p}\cdot\frac1n\right)\le\varepsilon,\quad\text{for all }n\ge\frac{d\cdot2^{\frac{2d}{1-p}}}{(1-p)(\exp(\varepsilon)-1)}.
    \end{align*}
    For $p<1/2$, we see that there exists $\nu\in(n^{d/(1-p)}-1,n^{d/(1-p)})$ and $\xi\in(n,n+1)$ such that
    \begin{align*}
     (1-2p)\cdot\sum_{j=f(n)}^{f(n+1)}\lambda_j^2&=\left\lceil(n+1)^{d/(1-p)}-1\right\rceil^{1-2p}-\left\lceil n^{d/(1-p)}-1\right\rceil^{1-2p}\\
				&\le(n+1)^\frac{d(1-2p)}{1-p}-\left(n^{d/(1-p)}-1\right)^{1-2p}\\
				&=(n+1)^\frac{d(1-2p)}{1-p}-n^{\frac{d(1-2p)}{1-p}}+(1-2p)\nu^{-2p}\\
				&\le(n+1)^\frac{d(1-2p)}{1-p}-n^{\frac{d(1-2p)}{1-p}}+(1-2p)(n^{d/(1-p)}-1)^{-2p}\\
				&=n^\frac{d(1-2p)}{1-p}+\frac{d(1-2p)}{1-p}\xi^{\frac{d(1-2p)}{1-p}-1}-n^{\frac{d(1-2p)}{1-p}}+(1-2p)(n^{d/(1-p)}-1)^{-2p}\\
				&\le\frac{d(1-2p)}{1-p}n^{\frac{d(1-2p)}{1-p}-1}+(1-2p)(n^{d/(1-p)}-1)^{-2p}
    \end{align*}
    Observe that $\frac{d(1-2p)-1+p}{1-p}=\frac{d(1-2p)}{1-p}-1<0$ since $d<\frac{1-p}{1-2p}$. Therefore,
    \begin{align*}
     \sum_{j=f(n)}^{f(n+1)}\lambda_j^2&\le\frac{d}{1-p}n^{\frac{d(1-2p)}{1-p}-1}+(n^{d/(1-p)}-1)^{-2p}\\
     &\le\varepsilon,\quad\text{for all }n\ge\max\left\{\left(\frac{2d}{(1-p)\varepsilon}\right)^{\frac{1-p}{d(1-2p)-1+p}},\left((2/\varepsilon)^{\frac{1}{2p}}+1\right)^{(1-p)/d}\right\}.
    \end{align*}
    Observing that $\theta_n=n^{-q}$ converges to $0$ with modulus $\sqrt[q]{1/\varepsilon}$, we summarize the moduli for this choice of the parameter sequences.
    \begin{enumerate}
     \item $\displaystyle n_0:=1$ and $\delta:=\frac{d(1-q)}{1-p}-1$,
     \item $f(n):=\lceil n^{d/(1-p)}\rceil$, where $d:=\min\left\{\frac{1-p}{1-r-p},\frac{3}{2}\left(1-\frac{q}{1-p}\right)^{-1}\right\}$,
     \item $\displaystyle\varphi_1(\varepsilon):=\sqrt[q]{1/\varepsilon}$,
     \item $\displaystyle\varphi_2(\varepsilon):=\left\lceil\left(\frac{2^{d}dq(d+1-p)}{\varepsilon(1-p)^2}\right)^2\right\rceil+1$,
     \item $\displaystyle\varphi_3(\varepsilon):=\left\{\begin{array}{ll}
					    ((2p-1)\varepsilon)^{\frac{1-p}{d(1-2p)}}+1&\text{for }p>\frac12,\\
					    \frac{d\cdot2^{\frac{2d}{1-p}}}{(1-p)(\exp(\varepsilon)-1)}+1&\text{for }p=\frac12,\\
                                            \max\left\{\left(\frac{2d}{(1-p)\varepsilon}\right)^{\frac{1-p}{d(1-2p)-1+p}},\left((2/\varepsilon)^{\frac{1}{2p}}+1\right)^{(1-p)/d}\right\}+1&\text{for }p<\frac12.
                                           \end{array}\right.$
    \end{enumerate}

    \subsection{Example 2}
    We begin with the following well-known inequality, whose proof we include for completeness.
    \begin{lemma}\label{lem:log}
     For all $x\ge0$, $\log(1+x)\le\frac{x}{\sqrt{1+x}}$.
    \end{lemma}
    \begin{proof}
     Define $f:[0,\infty]\to\RR$ by $f(x):=\frac{x}{\sqrt{1+x}}-\log(1+x)$. Then
     \begin{equation*}
      f'(x)=\frac{\sqrt{1+x}-\frac{x}{2\sqrt{1+x}}}{1+x}-\frac{1}{1+x}=\frac{\frac{2+2x-x}{\sqrt{1+x}}-1}{1+x}=\frac{2+x-\sqrt{1+x}}{(1+x)^{3/2}}\ge0.
     \end{equation*}
     Moreover, $f(0)=0$, so $f(x)\ge0$ for all $x\ge0$, whence the claim follows.
    \end{proof}

    Set $\lambda_n=1/n$ and $\theta_n=1/\log\log n$ for $n\ge 3$ and $\lambda_1=\lambda_2=\theta_1=\theta_2=0$ (see \cite{Bruck(74)}). Then, we may take $n_0:=3$, $\delta:=1/2$, $f(n):=n^n$, $\varphi_1(\varepsilon):=\exp\exp(1/\varepsilon)$, $\varphi_2:=\max\{e^4,\exp((1/\varepsilon)^2-1)-1\}$ and $\varphi_3:=\max\{3,\log(2/\varepsilon+1)\}$. That $\varphi_1$ is as required is immediate. Moreover, by Example 1 of \cite{Bruck(74)},
    \begin{equation*}
     \theta_{f(n)}\sum_{j=f(n)}^{f(n+1)}\lambda_j\ge\frac{\log n}{\log n+\log\log n}\ge\frac{\log n}{2\log n}=\frac12,\quad\text{for all }n\ge3,
    \end{equation*}
    so $n_0$ and $\delta$ are as required.
    
    Again from \cite{Bruck(74)},
    \begin{equation*}
     1+\log(n+1)\ge\sum_{j=f(n)}^{f(n+1)}\lambda_j\ge\log n,\quad\text{for all }n\ge3.
    \end{equation*}
    Moreover, for $n\ge3$,
    \begin{align*}
     \log\log (n+1)^{n+1}-\log\log n^n&\le\log\log\frac{(n+1)^{n+1}}{n^n}=\log\log\left(\left(\frac{n+1}{n}\right)^n(n+1)\right)\\
     &\le\log\log(e(n+1)).
    \end{align*}
    Consequently
    \begin{align*}
     \theta_{f(n)}-\theta_{f(n+1)}&=\frac{1}{\log\log n^n}-\frac{1}{\log\log (n+1)^{n+1}}\\
	      &=\frac{\log\log (n+1)^{n+1}-\log\log n^n}{\left(\log\log n^n\right)\cdot\left(\log\log(n+1)^{n+1}\right)}\\
	      &\le\frac{\log\log(e(n+1))}{\left(\log\log n^n\right)\cdot\left(\log\log(n+1)^{n+1}\right)}.
    \end{align*}
    Now $e^4\log e^4=4e^4=3e^4+e^4>e^5+e=e(e^4+1)$, so $\log(n\log n)\ge\log(e(n+1))$ for all $n\ge e^4$. Consequently, 
    \begin{align*}
     \left(\theta_{f(n)}-\theta_{f(n+1)}\right)\cdot\sum_{j=f(n)}^{f(n+1)}\lambda_j&\le\frac{1+\log(n+1)}{\log(n\log n)}\cdot\frac{\log\log(e(n+1))}{\log((n+1)\log(n+1))}\\
     &\le\frac{\log(e(n+1))}{\log(n\log n)}\cdot\frac{\log(1+\log(n+1))}{\log(n+1)}\\
     &\le 1\cdot\frac{\log(1+\log(n+1))}{\log(n+1)}.
    \end{align*}
    Now, we apply Lemma \ref{lem:log}
    \begin{align*}
    \left(\theta_{f(n)}-\theta_{f(n+1)}\right)\cdot\sum_{j=f(n)}^{f(n+1)}\lambda_j &\le\frac{\log(n+1)}{\log(n+1)\sqrt{1+\log(n+1)}}\\
     &=\frac{1}{\sqrt{1+\log(n+1)}}\\
     &\le\varepsilon,\quad\text{for all }n\ge\max\{e^4,\exp((1/\varepsilon)^2-1)-1\}.
    \end{align*}
    
    Moreover,
    \begin{align*}
     \sum_{j=f(n)}^{f(n+1)}\lambda_j^2\le\int_{j=f(n)-1}^{f(n+1)-1}j^{-2}dj&=-2\left(\left((n+1)^{(n+1)}-1\right)^{-1}-\left(n^n-1\right)^{-1}\right)\\
     &=\frac{2}{n^n-1}-\frac{2}{(n+1)^{(n+1)}-1}\\
     &\le\frac{2}{n^n-1}\le\varepsilon,\quad\text{for all }n\ge\max\{3,\log(2/\varepsilon+1)\}.
    \end{align*}

\end{document}